\documentclass[a4paper,11pt,oneside]{amsart}
\usepackage[T1]{fontenc}
\usepackage[utf8]{inputenc}
\usepackage{textcomp}
\usepackage[sc,noBBpl]{mathpazo} 
\usepackage{eulervm}
\usepackage{microtype}
\usepackage{ellipsis}
\usepackage[numbers,sort]{natbib}
\usepackage{amssymb}
\usepackage{fancyvrb}
\usepackage{todonotes}
\usepackage{hyperref}
\hypersetup{
  pdfauthor={Hans-Peter Schröcker},%
  pdftitle={The Bäcklund Transform of Principal Contact Element Nets},%
  pdfkeywords={principal constant element net, Gaussian curvature, pseudosphere, Bäcklund transformation, Bianchi Permutation Theorem},%
  pdfstartview={FitH},%
  pdfproducer={},%
  pdfcreator={},%
  pdfborder={0 0 0},%
  colorlinks=false,%
}

\renewcommand*{\vec}[1]{\mathbf{#1}}
\newcommand*{\NSet}{\mathbb{N}}
\newcommand*{\RSet}{\mathbb{R}}
\newcommand*{\ZSet}{\mathbb{Z}}
\newcommand*{\SQ}{\mathcal{S}}
\newcommand*{\SE}[1][3]{\mathrm{SE}(3)}
\newcommand*{\lset}[1]{\mathcal{#1}}
\newcommand*{\Norm}[1]{\Vert #1 \Vert}
\newcommand*{\identity}{\mathrm{id.}}

\newcommand*{\primal}[1]{\hat{#1}}
\newcommand*{\dual}[1]{\tilde{#1}}
\newcommand*{\dunit}{\varepsilon}

\newcommand*{\vpart}[1]{\mathfrak{v}{#1}}
\newcommand*{\conj}[1]{\bar{#1}}
\newcommand*{\econj}[1]{{#1}_\dunit}

\DeclareMathOperator*{\dist}{dist}
\DeclareMathOperator*{\discrim}{discrim}
\DeclareMathOperator*{\twist}{twist}

\newtheorem{theorem}{Theorem}
\newtheorem{lemma}[theorem]{Lemma}
\newtheorem{proposition}[theorem]{Proposition}
\newtheorem{corollary}[theorem]{Corollary}
\theoremstyle{definition}
\newtheorem{definition}[theorem]{Definition}
\theoremstyle{remark}
\newtheorem{remark}[theorem]{Remark}
\newtheorem{example}[theorem]{Example}

\begin{document}

\author{Hans-Peter Schröcker}
\title[The Bäcklund Transform of Principal Contact Element Nets]
      {The Bäcklund Transform of\\Principal Contact Element Nets}
\address{Hans-Peter Schröcker, University Innsbruck, Unit Geometry and
  CAD, Technikerstraße 20, A6020 Innsbruck, Austria}
\email{hans-peter.schroecker@uibk.ac.at}
\urladdr{http://geometrie.uibk.ac.at/schroecker/}

\begin{abstract}
We investigate geometric aspects of the the Bäcklund transform of
principal contact element nets. A Bäcklund transform exists if and
only if it the principal contact element net is of constant negative
Gaussian curvature (a pseudosphere). We describe an elementary
construction of the Bäcklund transform and prove its
correctness. Finally, we show that Bianchi's Permutation Theorem
remains valid in our discrete setting.


%
\end{abstract}
\maketitle

\section{Introduction}
\label{sec:introduction}

In the 1880s A.\;V.\;Bäcklund and L.\,Bianchi explored the
transformation of one surface of constant negative Gaussian curvature
(a \emph{pseudosphere}) into a surface of the same constant Gaussian
curvature. These surface transformations were named after Bäcklund
and, in their formulation in the language of partial differential
equations, play an important role in soliton theory and in integrable
systems. We refer the reader to the monograph
\cite{rogers02:_baecklund_darboux} for a comprehensive modern
treatment with many applications.

In 1952, W.\;Wunderlich gave a geometric description of the Bäcklund
transform of a discrete structure that is nowadays called a
K-net\,---\,a discrete asymptotic net of constant negative Gaussian
curvature \cite{wunderlich51:_differenzengeometrie}. Later, analytic
formulations were added and the discrete transformations were embedded
into the theory of Discrete Differential Geometry, see the monograph
\cite{bobenko08:_discrete_differential_geometry}.

In this article we perform a comprehensive geometric study of the
Bäcklund transform of principal contact element nets\,---\,nets of
contact elements such that any two neighbouring contact elements have
a common tangent sphere. Our main results are original but there exist
related contributions in the above mentioned publications. We also
have to mention the article \cite{schief03:_unification} by
W.\;Schief. It contains a description of the Bäcklund transform of
discrete O-surfaces. While the vertex sets of discrete O-surfaces and
principal contact element nets are identical (both are circular nets),
the normals differ. For O-surfaces, they are defined by a simple
algebraic condition, for principal contact element nets they satisfy a
geometric criterion. Accordingly, our approach is of geometric nature
as opposed to Schief's analytic treatment.

Our results are natural extensions of recent works in discrete
kinematics \cite{schroecker10:_four_positions_theory,%
  schroecker10:_discrete_gliding}. While their formulation is rather
natural at this point, some of our proofs tend to be rather involved
and require extensive calculations. We partly attribute this to the
observation that a (smooth or discrete) asymptotic parametrization is
more appropriate for the description of Bäcklund transforms. Principal
contact element nets discretize principal parametrizations and lead to
spatial structures of complicated nature.

The occasional use of a computer algebra system (CAS) is indispensable
in this work. We use Maple~13 for this purpose. In order to make the
computer calculations understandable, we try to give rather explicit
descriptions.

\section{Preliminaries and statement of main results}
\label{sec:preliminaries}

We proceed by giving definitions for the main concepts, by stating our
central results and by clarifying their relations. Proofs are deferred
to later sections.

Principal contact element nets provide a rich discrete surface
representation, consisting of points and normal vectors. They have
been introduced in \cite{bobenko07:organizing_principles} in an
attempt to develop a common master theory for circular nets (see
\cite[Section~3.5]{bobenko08:_discrete_differential_geometry}) and
conical nets (see
\cite[Section~3.4]{bobenko08:_discrete_differential_geometry} and
\cite{pottmann08:_focal_circular_conical}).

\begin{definition}
  \label{def:contact-element}
  A \emph{contact element} is a pair $(p,\vec{n})$ consisting of a
  point $a \in E^3$ (Euclidean three-space) and a unit vector $\vec{n}
  \in \RSet^3$. Its normal line $N$ is the oriented straight line
  through $a$ and in direction of $\vec{n}$, its oriented tangent
  plane $\pi$ is the plane through $p$ with oriented normal
  vector~$\vec{n}$.
\end{definition}

In this text we will always denote normal vector and normal line by
the same letter but in different fonts: Lowercase, boldface for
vectors, uppercase, italic for lines\,---\,just as in
Definition~\ref{def:contact-element}.

\begin{definition}
  \label{def:pcen}
  A \emph{contact element net} is a map $(p,\vec{n})$ from $\ZSet^d$,
  $d \ge 2$, to the space of contact elements. A \emph{principal
    contact element net} is a contact element net such that any two
  neighbouring contact elements $(p_i,\vec{n}_i)$ and
  $(p_j,\vec{n}_j)$ have a common oriented tangent sphere.
\end{definition}

The defining condition of principal contact element nets is rather
restrictive. It implies that neighbouring normal lines $N_i$ and $N_j$
intersect in a common point $h_{ij}$ which is at the same oriented
distance from both vertices $p_i$ and $p_j$. This shows that
neighbouring contact elements have a bisector plane
$\beta_{ij}$. Moreover, the vertices of an elementary quadrilateral
lie on a circle, the tangent planes are tangent to a cone of
revolution and the normal lines form a skew quadrilateral on a
hyperboloid of revolution.

An elementary quadrilateral $(p_0,\vec{n}_0)$, $(p_1,\vec{n}_1)$,
$(p_2,\vec{n}_2)$, $(p_3,\vec{n}_3)$ of a principal contact element
net can be constructed by choosing the vertices $p_0$, $p_1$, $p_2$,
and $p_3$ on a circle, prescribin $\vec{n}_0$ and then finding the
other normal vectors by reflection in the bisector planes
$\beta_{i,i+1}$ of $p_i$ and $p_{i+1}$.

In this article we investigate principal contact element nets that
admit a Bäcklund transform:

\begin{definition}
  \label{def:baecklund-transform}
  Two principal contact element nets $(p,\vec{n})$ and $(q,\vec{m})$
  are called \emph{Bäcklund mates} if
  \begin{enumerate}
  \item The distance of corresponding points $p_i$ and $q_i$ is
    constant.
  \item The angle of corresponding normals $\vec{n}_i$ and $\vec{m}_i$
    is constant. (It is necessary to measure this angle consistently
    in counter-clockwise direction when viewed along the ray from
    $p_i$ to~$q_i$.)
  \item Two corresponding tangent planes $\pi_i$ and $\chi_i$
    intersect in the connecting line $p_i \vee q_i$ of their vertices.
  \end{enumerate}
  In this case, $(q,\vec{m})$ is also called a \emph{Bäcklund
    transform} of $(p,\vec{n})$ (see
  Figure~\ref{fig:baecklund-mates}).
\end{definition}

\begin{figure}
  \centering
  \includegraphics{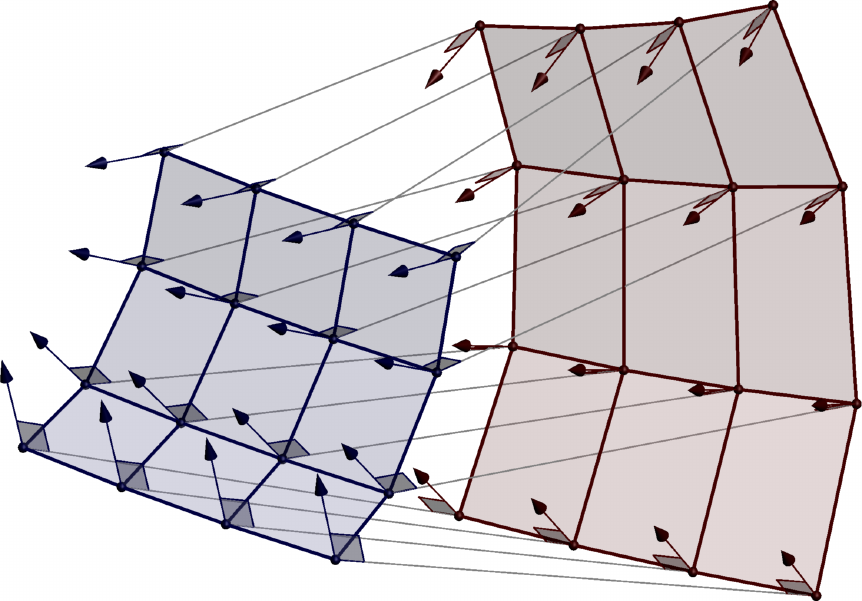}
  \caption{Bäcklund mates}
  \label{fig:baecklund-mates}
\end{figure}

The last condition in Definition~\ref{def:baecklund-transform} could
be replaced by the requirement that the connecting line $p_i \vee q_i$
is perpendicular to both normal vectors, $\vec{n}_i$ and~$\vec{m}_i$.

In the classical theory, Bäcklund transforms of smooth surfaces
$\Phi$, $\Psi$ can be defined by the same three conditions as in
Definition~\ref{def:baecklund-transform} (with the understanding that
both surfaces are parametrized over the same domain and corresponding
points, normals, and tangent planes belong to the same parameter
values). It is our aim in this article to prove several results on
Bäcklund transforms of principal contact element nets that are also
valid in the smooth theory and in other discrete settings.  The most
important property states that only pseudospheres admit Bäcklund
transforms.

\begin{theorem}
  \label{th:gaussian-curvature}
  If a principal contact element net $(p,\vec{n})$ admits a Bäcklund
  transform, it is of constant negative Gaussian curvature
  \begin{equation}
    \label{eq:1}
    K = -\frac{\sin^2\varphi}{d}
  \end{equation}
  where $d$ is the distance between corresponding points and $\varphi$
  is the angle between corresponding normals.
\end{theorem}

We still have to explain the notion of Gaussian curvature to which
this theorem refers.  In the smooth setting, the Gaussian curvature in
a point can be defined as the local area distortion under the Gauss
map. This is imitated in

\begin{definition}
  \label{def:gaussian-curvature}
  The Gaussian curvature $K$ of an elementary quadrilateral
  $(p_0,\vec{n}_0)$, $(p_1,\vec{n}_1)$, $(p_2,\vec{n}_2)$,
  $(p_3,\vec{n}_3)$ of a principal contact element net is defined as
  \begin{equation}
    \label{eq:2}
    K = \frac{S_0}{S}
  \end{equation}
  where $S_0$ is the oriented area of the spherical quadrilateral
  $\vec{n}_0$, $\vec{n}_1$, $\vec{n}_2$, $\vec{n}_3$ and $S$ is the
  oriented area of the circular quadrilateral $p_0$, $p_1$,
  $p_2$,~$p_3$.
\end{definition}

The Gaussian curvature as defined here is a well-accepted concept in
discrete differential geometry (see
\cite{bobenko08:_discrete_differential_geometry}) and it is a special
case of a recently published and more general theory
\cite{bobenko10:_curvature_theory}.  Note that in contrast to the
smooth setting, the Gaussian curvature is assigned to a face and not
to a vertex. This allows us to speak of the Gaussian curvature of an
elementary quadrilateral.

It will be convenient to speak of the twist of two oriented lines and
also of the twist of Bäcklund mates.

\begin{definition}
  \label{def:twist}
  The \emph{twist} of two oriented lines $N$ and $M$ is defined as the
  ratio
  \begin{equation}
    \label{eq:3}
    \twist(N,M) = \frac{\sin\varphi}{d}
  \end{equation}
  where $\varphi$ is the angle between the respective direction
  vectors $\vec{n}$ and $\vec{m}$ (measured according to the
  convention of Definition~\ref{def:baecklund-transform}) and $d$ is
  the distance between $N$ and $D$
  (Figure~\ref{fig:transversals}). The twist of two \emph{Bäcklund
    mates} is defined as the twist of any two corresponding normal
  lines.
\end{definition}

\begin{figure}
  \centering
  \includegraphics{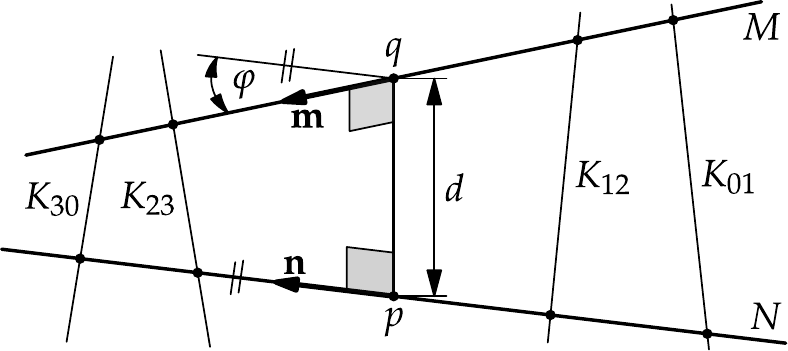}
  \caption{Transversals of four relative revolute axes}
  \label{fig:transversals}
\end{figure}

We will argue that Theorem~\ref{th:gaussian-curvature} is a
consequence of the Theorem~\ref{th:rotation-quadrilateral}, below,
which deserves interest in its own. In order to formulate this result,
we recall a definition from
\cite{schroecker10:_four_positions_theory}:

\begin{definition}
  \label{def:rotation-quadrilateral}
  A cyclic sequence $A_0$, $A_1$, $A_2$, $A_3$ of direct Euclidean
  displacements $A_i\colon E^3 \to E^3$ is called a \emph{rotation
    quadrilateral,} if any two neighboring positions correspond in a
  rotation.
\end{definition}

A generic rotation quadrilateral has four relative revolute axes
$K_{01}$, $K_{12}$, $K_{23}$, $K_{30}$ (in the moving space) that
admit two transversal lines $N$ and $M$. Denote the feet of their
common perpendicular by $p \in N$ and $q \in M$ and one of their unit
direction vectors by $\vec{n}$ and $\vec{m}$ (thus imposing an
orientation on $N$ and $M$). It was shown in
\cite{schroecker10:_four_positions_theory} that precisely the
homologous images of the two contact elements $(p,\vec{n})$ and
$(q,\vec{m})$ can serve as corresponding elementary quadrilaterals of
Bäcklund mates. This leads us to an important result that relates the
Bäcklund transform of pseudospherical principal contact element nets
to discrete rotating motions.

\begin{theorem}
  \label{th:rotation-quadrilateral}
  Consider a rotation quadrilateral $A_0$, $A_1$, $A_2$, $A_3$ with
  relative revolute axes $K_{01}$, $K_{12}$, $K_{23}$, $K_{30}$ in the
  moving space. Denote by $N$ and $M$ their transversals, by $\vec{n}$
  and $\vec{m}$ one of their respective unit direction vectors and by
  $p \in N$ and $q \in M$ the feet of their common normals
  (Figure~\ref{fig:transversals}). Then the Gaussian curvature of the
  homologous images of $(p,\vec{n})$ and $(q,\vec{m})$ equals
  $-\twist^2(N,M)$, that is, it can be computed by \eqref{eq:1} where
  $d = \dist(p,q)$ and $\varphi = \sphericalangle(\vec{n},\vec{m})$.
\end{theorem}

As a consequence of Theorem~\ref{th:rotation-quadrilateral} a
pseudospherical principal contact element net of Gaussian curvature
$K$ admits Bäcklund transforms of twist $\pm\sqrt{-K}$.

In retrospect and especially when studying its relation to
Theorem~\ref{th:gaussian-curvature}, the content of
Theorem~\ref{th:rotation-quadrilateral} is maybe a rather obvious
conjecture. Nonetheless, it is surprising when viewed
independently. There are six degrees of freedom when fitting a
rotation quadrilateral to two given lines $N$ and $M$: Three choices
for every relative rotation makes a total of twelve degrees of
freedom; the requirement that their composition yields the identity
consumes six of them. Still, this is insufficient to change the area
ratio \eqref{eq:2}.

Our proof of Theorem~\ref{th:rotation-quadrilateral} uses the dual
quaternion calculus of rigid motions. It allows a computational proof
that requires only linear operations. Nonetheless, the involved
calculations are so complicated that assistance of a CAS seems
indispensable. The reason for this is the rather involved spatial
configuration and the high number of degrees of
freedom. Theorems~\ref{th:gaussian-curvature} and
\ref{th:rotation-quadrilateral} will both be proved in
Section~\ref{sec:gaussian-curvature}.

A certain converse of Theorem~\ref{th:gaussian-curvature} is the
construction of the Bäcklund transform from a given pseudospherical
principal contact element net. It is on the agenda in
Section~\ref{sec:construction}.

\begin{theorem}
  \label{th:baecklund-construction}
  Given a pseudospherical principal contact element net $(p,\vec{n})$
  of negative Gaussian curvature $K$ and a contact element
  $(q_i,\vec{m}_i)$ such that
  \begin{itemize}
  \item $q_i$ lies in the tangent plane $\pi_i$ of $p_i$,
  \item $\vec{m}_i$ is perpendicular to the line $p_i \vee q_i$, and
  \item the quantities $K$, $d = \dist(p_i,q_i)$ and $\varphi =
    \sphericalangle(\vec{n}_i,\vec{m}_i)$ satisfy \eqref{eq:1}
  \end{itemize}
  there exists precisely one Bäcklund transform $(q,\vec{m})$ of
  $(p,\vec{n})$ such that $(p_i,\vec{n}_i)$ and $(q_i,\vec{m}_i)$
  correspond.
\end{theorem}

As we will see, there exists a simple necessary construction for the
neighbours of the contact element $(q_i,\vec{m}_i)$. This already
implies uniqueness of the Bäcklund transform. In order to prove
existence, we have to show that the conditions on $q_i$ and
$\vec{m}_i$ imply that this construction does not lead to
contradictions. The proof consists once more of a direct CAS-aided
computation.

One interpretation of Theorem~\ref{th:baecklund-construction} is that
every pseudospherical principal contact element net can be generated
in infinitely many ways as trajectory of a discrete rotating motion
that has a second, non-parallel trajectory surface of the same type.
Combining this observation with results of
\cite{schroecker10:_discrete_gliding}, in particular with
\cite[Theorem~7]{schroecker10:_discrete_gliding}, we get an
interesting corollary (the terminology is explained in the standard
reference \cite{bobenko08:_discrete_differential_geometry} on discrete
differential geometry):

\begin{corollary}
  \label{cor:consistency}
  Principal contact element nets of a prescribed negative Gaussian
  curvature are described by a multidimensionally consistent
  $2D$-system.
\end{corollary}

This corollary certainly admits simple direct proofs. Still, we are
not aware of a reference that actually states it.

We should also mention that our results are in perfect analogy to the
smooth setting. If $\Phi$ and $\Psi$ are Bäcklund mates, the
correspondence between their points can be realized by principal
parametrizations $\vec{x}(u,v)$ and $\vec{y}(u,v)$. For every
parameter pair $(u,v)$, we obtain a rigid figure consisting of
corresponding points $\vec{x}(u,v)$, $\vec{y}(u,v)$ and the respective
surface normals. This induces a two-parametric motion $A(u,v)$ with
the following properties:
\begin{itemize}
\item $A(u,v)$ is a gliding motion on both surfaces $\Phi$ and $\Psi$
  \cite[Section~7.1.5]{pottmann01:_line_geometry}. This means that
  there exist two planes in the moving space whose images under
  $A(u,v)$ envelop $\Phi$ and $\Psi$, respectively.
\item The infinitesimal motions in $u$- and $v$-parameter directions
  are infinitesimal rotations. (This is a general property of gliding
  motions along principally parametrized surfaces.)
\end{itemize}
In our discrete setting, both properties are preserved. The
infinitesimal rotations of the continuous case are replaced by
rotations between finitely separated positions.

Our final result, which will be presented in more detail in
Section~\ref{sec:permutation}, is Bianchi's Permutation Theorem for
Bäcklund transforms of pseudospherical principal contact element nets:

\begin{theorem}
  \label{th:bianchi-permutation}
  Suppose $(b,\vec{l})$ and $(c,\vec{m})$ are two Bäcklund transforms
  of the same twist to a principal contact element net
  $(a,\vec{k})$. Then there exists a unique principal contact element
  net $(d,\vec{n})$ which is at the same time a Bäcklund transform of
  $(b,\vec{l})$ and~$(c,\vec{m})$.
\end{theorem}

Note that throughout this article we implicitely make a number of
assumptions on the genericity of the involved geometric entities.
More specifically, we require that the Gaussian curvature \eqref{eq:2}
is well-defined for all elementary quadrilaterals of principal contact
element nets (with the notable exception of
Example~\ref{ex:discrete-tractrix}. The meaning of the word
``generic'' in the paragraph folling
Definition~\ref{def:rotation-quadrilateral} is not as easily
explained. We require that indeed two transversals to the four
relative revolute axes exist in algebraic sense. This can be
formulated as a non-vanishing condition of an algebraic expression in
the input paramters. Moreover, neighbouring relative revolute axes are
assumed to be skew.

\section{The Gaussian curvature of Bäcklund mates}
\label{sec:gaussian-curvature}

This section is dedicated to the proof of
Theorem~\ref{th:gaussian-curvature}. Assume that $(p,\vec{n})$ and
$(q,\vec{m})$ are Bäcklund mates, denote the distance of corresponding
points by $d$ and the angle of corresponding normals by $\varphi$. We
want to show that the Gaussian curvature of all elementary
quadrilaterals of $(p,\vec{n})$ is constant and can be computed by
\eqref{eq:1}.

Our first observation concerns two neighbouring pairs of corresponding
contact elements.

\begin{lemma}
  \label{lem:rotation}
  Under the conditions of Theorem~\ref{th:gaussian-curvature}, any two
  neighbouring pairs $(p_0,\vec{n}_0)$, $(p_1,\vec{n}_1)$ and
  $(q_0,\vec{m}_0)$, $(q_1,\vec{m}_1)$ of corresponding contact
  elements correspond in a rotation~$R_{01}$ about the line $L_{01} =
  h_{01} \vee k_{01}$ where $h_{01} = N_0 \cap N_1$ and $k_{01} = M_0
  \cap M_1$.
\end{lemma}

\begin{proof}
  By Definition~\ref{def:baecklund-transform}, the figures formed by
  corresponding contact elements are congruent and of equal
  orientation. Thus, there exists a direct Euclidean displacement
  $R_{01}$ such that $R_{01}(p_0,\vec{n}_0) = (p_1,\vec{n}_1)$ and
  $R_{01}(q_0,\vec{m}_0) = (q_1,\vec{m}_1)$. By
  Definition~\ref{def:pcen} the points $h_{01}$ and $k_{01}$ are
  fix-points of $R_{01}$. Thus, $R_{01}$ is indeed a rotation about
  the line $L_{01}$.
\end{proof}

Consider now corresponding elementary quadrilaterals
$(p_0,\vec{n}_0)$, \ldots, $(p_3,\vec{n}_3)$, and $(q_0,\vec{m}_0)$,
\ldots, $(q_3,\vec{m}_3)$ of the two Bäcklund mates. They define four
relative rotations $R_{01}$, $R_{12}$, $R_{23}$, $R_{30}$ such that
\begin{equation}
  \label{eq:4}
  R_{30} \circ R_{23} \circ R_{12} \circ R_{01} = \identity
\end{equation}
This observation relates the Bäcklund transform of pseudospherical
principal contact element nets to the geometry of rotation
quadrilaterals \cite{schroecker10:_four_positions_theory} and discrete
gliding motions \cite{schroecker10:_discrete_gliding}. We will use
results and techniques from both articles for our proof.

More specifically, we view each figure $(p_i,\vec{n}_i)$,
$(q_i,\vec{m}_i)$ as the position of a rigid body in space. This gives
us four direct Euclidean displacements
\begin{equation}
  \label{eq:5}
  A_i\colon E^3 \to E^3
\end{equation}
from the moving space $E^3$ to an identical copy of $E^3$, the fixed
space. They form a rotation quadrilateral in the sense of
Definition~\ref{def:rotation-quadrilateral}.
Theorem~\ref{th:rotation-quadrilateral} states that the Gaussian
curvature of an elementary quadrilateral depends only on the relative
position of two corresponding normal lines. Since this is the same for
all elementary quadrilaterals, Theorem~\ref{th:gaussian-curvature}
follows.

In our proof of Theorem~\ref{th:rotation-quadrilateral} we use the
dual quaternion calculus of spatial kinematics. It is explained for
example in \cite{husty09:_algebraic_geometry_kinematics}. A short
primer is also given in \cite{schroecker10:_discrete_gliding}. We use
the notation of the latter article. A direct Euclidean displacement is
modelled as a unit dual quaternion $A = \primal{A}+\dunit\dual{A}$
where $\primal{A}$ and $\dual{A}$ are its respective primal and dual
part and $\dunit$ satisifes $\dunit^2 = 0$.  Assembling the components
of primal and dual parts into a homogeneous vector $[\primal{A},
\dual{A}] = [a_0,\ldots,a_7]$, we can view $A$ as a point on the
\begin{equation}
  \label{eq:50}
  \SQ \subset P^7\colon \primal{A} \cdot \dual{A} = 0.
\end{equation}
This quadric is called the \emph{Study quadric.}

\begin{proof}[Proof of Theorem~\ref{th:rotation-quadrilateral}]
  We start by setting up a system of equations that describe a
  rotation quadrilateral whose relative rotation axes in the moving
  space intersect two fixed lines $N$ and $M$. At first, we specify
  the two lines $N$ and $M$ in the moving space. Without loss of
  generality, the normal feet $p$ and $q$ on $N$ and $M$,
  respectively, can be chosen as
  \begin{equation}
    \label{eq:6}
    p = (0, 0, e),\quad
    q = (0, 0, -e)
  \end{equation}
  where $2e = d$. The respective directions vectors of $N$ and $M$ are
  \begin{equation}
    \label{eq:7}
    \vec{n} = (1, t, 0),\quad
    \vec{m} = (1, -t, 0)
  \end{equation}
  where $t$ is related to the angle $\varphi$ between $N$ and $M$ via
  $\varphi = 2 \arctan t$. Note that the vectors $\vec{n}$ and
  $\vec{m}$ are not normalized so that we ultimately will not compute
  the Gaussian image of an elementary quadrilateral on the unit sphere
  but on a sphere of squared radius $1+t^2$. This is admissible since
  the Gaussian curvature will only be multiplied by the constant
  factor $1+t^2$.

  Without loss of generality, the first position of the rotation
  quadrilateral can be taken as the identity, $A_0 = [1,0,\ldots,0]$
  (we use homogeneous coordinates in $P^7$). The two neighbouring
  positions are
  \begin{equation}
    \label{eq:8}
    A_1 = [a_{10}, \ldots, a_{17}],
    \quad\text{and}\quad
    A_3 = [a_{30}, \ldots, a_{37}].
  \end{equation}
  Their components $a_{ij}$ are subjects to seven constraints:
  \begin{itemize}
  \item The relative displacements $R_{01}$ and $R_{03}$ are
    rotations. In the dual quaternion calculus this constraint is
    modelled as
    \begin{equation}
      \label{eq:9}
      \pi_5(A_i \star A_0^{-1}) = 0,
      \quad
      i \in \{1, 3\},
    \end{equation}
    where $\pi_5$ denotes the projection onto the fifth coordinate.
  \item The relative revolute axes $K_{01}$ and $K_{03}$ intersect $M$
    and $N$:
    \begin{equation}
      \label{eq:10}
      \Omega(
        \econj{(R_{0i})} \star \vpart{\econj{(R_{0i} \star A_0)}} \star \econj{(R_{0i}^{-1})}, T
      ) = 0;
    \quad i \in \{1, 3\},\ T \in \{M, N\}.
    \end{equation}
    In this equation, $\Omega$ is the bilinear form associated to the
    Study quadric ($\Omega(X,Y) = \primal{X} \cdot \dual{Y}$), the
    subscript $\dunit$ denotes $\dunit$-conjugation of dual
    quaternions, and $\vpart{X}$ is the vector part of the dual
    quaternion~$X$.
  \item The sought positions lie on the Study quadric $\SQ$:
    \begin{equation}
      \label{eq:11}
      \Omega(A, A) = \primal{A}_i \cdot \dual{A}_i = 0,
      \quad
      i \in \{1,3\}.
    \end{equation}
  \end{itemize}
  Equations~\eqref{eq:9} and \eqref{eq:10} are linear in the
  components of $A_1$ and $A_3$, Equation~\eqref{eq:11} is quadratic.
  Normalizing by $a_{10} = ta_{13}$, the general solution for $A_1$
  reads
  \begin{equation}
    \label{eq:12}
    \begin{gathered}
      A_1 = (ta_{13}, ta_{11}a_{13}, ta_{13}^2, 0, -t^2a_{11}a_{13},
      a_{12}a_{13}, t^2a_{11}^2 - a_{12}^2);\\
      a_{11}, a_{12}, a_{13} \in \RSet.
    \end{gathered}
  \end{equation}
  Here, $a_{11}$, $a_{12}$, and $a_{13}$ serve as free parameters. The
  solution for $A_3$ is obtained by replacing the parameter $a_{1j}$
  with $a_{3j}$ ($j = 1,2,3$).

  Similarly, the missing position $A_2$ is uniquely determined by two
  linear equations $E_1$, $E_2$ of type \eqref{eq:9}, four linear
  equation $E_3$, \ldots, $E_6$ of type \eqref{eq:10} and one
  quadratic equation $E_7$ of type \eqref{eq:11}. Thus, it is the
  second intersection point of a straight line through $A_0$ with the
  Study quadric $\SQ$. The solution is unique and can be computed in
  rational arithmetic.

  While it is possible to compute $A_2$ and verify \eqref{eq:1}
  directly, the involved expressions are excessively long. Our
  implementation of this approach delivers the desired result but only
  after a few hours of computation. Therefore, we favor a method which
  is based on ideal theory and yields a confirmatory answer within a
  few minutes.\footnote{We thank Dominic Walter for his help with this
    approach.} The basic idea is to show that a certain polynomial
  (Equation~\eqref{eq:17}, below) is contained in the ideal spanned by
  the polynomial equations $E_1$, \ldots, $E_7$.

  For $i \in \{0,1,2,3\}$ we compute the images $p_i$ of the normal
  foot $p$ and the images $\vec{n}_i$ of the direction vector
  $\vec{n}$ under the displacement $A_i$ (with the entries $a_{2i}$ of
  $A_2$ left unspecified). This can be accomplished by using the
  following formulas for the action of a dual quaternion $A$ on a
  homogeneous coordinate vector $[x_0,x_1,x_2,x_3] = [x_0,\vec{x}]$:
  \begin{equation}
    \label{eq:13}
    x'_0 + \dunit \vec{x}' = \econj{A} \star (x_0 + \dunit\vec{x}) \star \conj{A}
  \end{equation}
  (the bar denotes dual quaternion conjugation). Note that the use of
  homogeneous coordinates avoids the introduction of
  denominators\,---\,this is highly desirable in CAS calculations of a
  large complexity\,---\,but has to be taken into account later in
  \eqref{eq:14} when computing oriented areas.

  Instead of computing the oriented area of the quadrilaterals $p_0$,
  $p_1$, $p_2$, $p_3$ and $\vec{n}_0$, $\vec{n}_1$, $\vec{n}_2$,
  $\vec{n}_3$ we compute the areas $S$ and $S_0$ of their respective
  projection onto the first and second coordinate. This is possible
  since an orthographic projection does not affect the ratio of areas
  in parallel planes. We have
  \begin{equation}
    \label{eq:14}
    \begin{gathered}
      S = \sum_{i=0}^3 (p_{i,1}p_{i+1,2} - p_{i,2}p_{i+1,1})\varkappa_{i+2}\varkappa_{i+3},\\
      S_0 = \sum_{i=0}^3 (\vec{n}_{i,1}\vec{n}_{i+1,2} - \vec{n}_{i,2}\vec{n}_{i+1,1})\varkappa_{i+2}\varkappa_{i+3}
    \end{gathered}
  \end{equation}
  (indices modulo four). The factors $\varkappa_i = \sum_{j=0}^3
  a_{ij}^2$ compensate the use of homogeneous coordinates in the
  computation of the points $p_i$ and the ideal points (vectors)
  $\vec{n}_i$. Now we have to show that
  \begin{equation}
    \label{eq:15}
    \frac{S_0}{S} = -\frac{t^2}{(1+t^2)e^2},
  \end{equation}
  or, equivalently,
  \begin{equation}
    \label{eq:16}
    (1+t^2)e^2S_0 - t^2S = 0.
  \end{equation}
  Equation~\ref{eq:16} is polynomial. It holds true if its left-hand
  side is contained in the ideal $J$ spanned by the seven equations
  $E_1$, \ldots, $E_7$ that determine the fourth position of
  $A_2$. The indeterminates are $a_{20}$, \ldots, $a_{27}$ while
  $a_{11}$, $a_{12}$, $a_{13}$, $a_{31}$, $a_{32}$, $a_{33}$, $e$, $t$
  serve as parameters. We perform the necessary algebraic
  manipulations by means of the CAS Maple~13 and assume that the
  defining equations of $A_2$ are stored in \verb|E1|, \ldots,
  \verb|E7|. Moreover, the simplifying normalization $a_{20} = 1$ is
  applied.
  \begin{Verbatim}[numbers=left]
with(Groebner):
# J defined by linear and quadratic equations
J := [E1, E2, E3, E4, E5, E6, E7]:

# Use tdeg term order:
TO := tdeg(a21, a22, a23, a24, a25, a26, a27): 

# Interreduce J and reduce R with respect to this reduced ideal.
# The result is indeed 0, showing that R is contained in J.
J := InterReduce(J, TO):
Reduce([R], J, TO)[1];
  \end{Verbatim}
  The final output is $0$, showing that \verb|R| is indeed contained
  in the ideal~\verb|J|.
\end{proof}

This proof also finishes the proof of
Theorem~\ref{th:gaussian-curvature}. A few explaining remarks
concerning the computer calculation seem appropriate:

\begin{itemize}
\item The \verb|tdeg|-term ordering (Line~6) is primarily by total
  degree and then reversed lexicographic.
\item Inter-reducing the generating elements of the ideal \verb|J|
  with respect to the term order \verb|TO| (Line~10) produces a list
  of polynomials that generate the same ideal \verb|J| but are reduced
  in the sense that no monomial is reducible by the leading monomial
  of another list element.
\item The command \verb|Reduce| (Line~11) computes the remainder of
  \verb|R| divided by the polynomials in \verb|J|. Its vanishing
  implies that \verb|R| is indeed contained in the ideal spanned by
  \verb|J|.
\end{itemize}

The running time of this implementation is approximately six
minutes. Minor improvements are possibly by adapting the order of the
variables in Line~6. We regret not being able to offer a proof of
Theorem~\ref{th:gaussian-curvature} that does not rely on a computer
algebra system. The main reason for the computational difficulties (in
spite of the fact that the problem is linear) is the high number of
six free parameters. A more insightful proof or a proof with manually
tractable calculations would be desirable.

\section{Construction of the Bäcklund transform}
\label{sec:construction}

In this section we provide an algorithm for actually constructing (or
computing) the Bäcklund transforms of a pseudospherical principal
contact element net. The proof of correctness of this algorithm is at
the same time the proof of Theorem~\ref{th:baecklund-construction}.

It is a simple observation that the Bäcklund transform $(q,\vec{m})$
in Theorem~\ref{th:baecklund-construction} is necessarily
unique. Existence is a different matter. Given the principal contact
element net $(p,\vec{n})$ and the contact element $(q_i,\vec{m}_i)$,
any neighbour $(q_j,\vec{m}_j)$ to $(q_i,\vec{m}_i)$ is uniquely
determined. It is found by the following steps whose necessity has
already been discussed:
\begin{itemize}
\item Determine the points $h_{ij} = N_i \cap N_j$ and $k_{ij} = M_i
  \cap \beta_{ij}$ where $\beta_{ij}$ is the bisector plane of $p_i$
  and $p_j$.
\item Determine the unique rotation $R_{ij}$ about the axis $L_{ij} =
  h_{ij} \vee k_{ij}$ that transforms $(p_i,\vec{n}_i)$ into
  $(p_j,\vec{n}_j)$.
\item The contact element $(q_j,\vec{m}_j)$ is the image of
  $(q_i,\vec{m}_i)$ under the rotation~$R_{ij}$.
\end{itemize}
We refer to this construction briefly as the ``neighbour
construction''. Note that a minor variation takes as input the
oriented lines $N_i$, $N_j$ and $M_i$ (as opposed to three contact
elements) and returns the oriented line $M_j$.

We provide some more information on the actual computation of the
neighbour construction\,---\,not only as a service for the reader but
also because we need them later in the proof of
Lemma~\ref{lem:projectivity} and
Theorem~\ref{th:baecklund-construction}.

Throughout this calculation we use homogeneous coordinates to describe
points and planes, Plücker coordinates to describe lines and
homogeneous four by four matrices to describe Euclidean
displacements. (Here, the use of dual quaternion calculus does not
provide a significant advantage.)  The bisector plane $\beta$ of two
non-ideal points $[x_0,x_1,x_2,x_3] = [x_0,\vec{x}]$ and
$[y_0,y_1,y_2,y_3]=[y_0,\vec{y}]$ has homogeneous coordinates
\begin{equation}
  \label{eq:17}
  [
    y_0^2 \;\vec{x} \cdot \vec{x} - x_0^2 \;\vec{y} \cdot \vec{y},
    2x_0y_0 \;\vec{x} \times \vec{y}
  ].
\end{equation}
The intersection point of the plane $[u_0,u_1,u_2,u_3] =
[u_0,\vec{u}]$ and the line with Plücker coordinates
$[g_1,g_2,g_3,h_1,h_2,h_3] = [\vec{g},\vec{h}]$ equals
\begin{equation}
  \label{eq:18}
  [x_0,\vec{x}] = [\vec{u} \cdot \vec{g}, -u_0\vec{g} + \vec{u} \times \vec{h}].
\end{equation}
Because of $h_{ij} = N_i \cap \beta_{ij} = N_j \cap \beta_{ij}$ this
formula can be used to determine both, $k_{ij}$ and $h_{ij}$.

The rotation $R_{ij}$ can be conveniently written as the composition
of two reflections in the bisector plane $\beta_{ij}$ and the plane
$\gamma_{ij} = p_j \vee h_{ij} \vee k_{ij}$. This is true since the
composition of two reflections in planes $\beta$ and $\gamma$ is a
rotation about the intersection line $L = \beta \cap
\gamma$. Moreover, it is obvious that $(p_i,\vec{n}_i)$ is transformed
into $(p_j,\vec{n}_j)$. The homogeneous transformation matrix
$[R_{ij}]$ of $R_{ij}$ is given by
\begin{equation}
  \label{eq:19}
  [R_{ij}] = [C_{ij}] \cdot [B_{ij}]
\end{equation}
where $[B_{ij}]$ and $[C_{ij}]$ are the reflection matrices in
$\beta_{ij}$ and $\gamma_{ij}$, respectively. They can be computed by
means of the general formula for the reflection matrix in a plane
$[u_0,u_1,u_2,u_3]$ which reads
\begin{equation}
  \label{eq:20}
  \begin{bmatrix}
    u_1^2 + u_2^2 + u_3^2 & 0                      & 0                     & 0        \\
    -2u_0u_1              & -u_1^2 + u_2^2 + u_3^2 & -2u_1u_2              & -2u_1u_3 \\
    -2u_0u_2              & -2u_1u_2               & u_1^2 - u_2^2 + u_3^2 & -2u_2u_3 \\
    -2u_0u_3              & -2u_1u_3               & -2u_2u_3              & u_1^2 + u_2^2 - u_3^2
  \end{bmatrix}.
\end{equation}
Equations~\eqref{eq:17}--\eqref{eq:20} are sufficient for a CAS
implementation of the neighbour construction.

We return to the proof of Theorem~\ref{th:baecklund-construction}
where existence is still open. The problem with the neighbour
construction is that it will produce contradictions for a general
choice of $(q_0,\vec{m}_0)$. Consider only an elementary quadrilateral
$(p_0,\vec{n}_0)$, \ldots, $(p_3,\vec{n}_3)$ and the contact element
$(q_0,\vec{m}_0)$. By the neighbour construction we obtain (in that
order) the contact elements $(q_1,\vec{m}_1)$, $(q_2,\vec{m}_2)$, and
$(q_3,\vec{m}_3)$. Applying the neighbour construction once more, we
get a contact element $(q^\star_0,\vec{m}^\star_0)$ which should equal
$(q_0,\vec{m}_0)$.

Since the Gaussian curvature $K$ and the distance $d$ between
corresponding points is already defined, we necessarily have to choose
$\vec{m}_0$ so that the angle $\varphi =
\sphericalangle(\vec{n}_0,\vec{m}_0)$ satisfies \eqref{eq:1}. Thus,
there are only four choices for the unit vector $\vec{m}_0$ and only
two choices for the normal line $M_0$. We will show that all of them
lead to valid solutions. This is also enough to ensure that the
neighbour construction can be consistently applied to the whole
principal contact element net. The proof of this still needs some
preparatory work.

To begin with, we mention those configuration that will be identified
as false positives by our test $(q_0,\vec{m}_0) = (q'_0,\vec{m}'_0)$.
\begin{itemize}
\item If $M_0$ intersects the axis of the circle through $p_0$, $p_1$,
  $p_2$, $p_3$, the composition of the four rotations yields the
  identity but the homologous contact elements $(q_0,\vec{m}_0)$,
  \ldots, $(q_3,\vec{m}_3)$ do not form the elementary quadrilateral
  of a principal contact element net.
\item If $M_0$ is parallel to $N_0$ the action of the rotation
  $R_{i,i+1}$ on $(q_i,\vec{m}_i)$ equals that of the reflection in
  $\beta_{i,i+1}$. Therefore, we get $(q^\star_0,\vec{m}^\star_0) =
  (q_0,\vec{m}_0)$ also in this case. The Gaussian curvature of the
  two elementary quadrilaterals is, however, different.
\end{itemize}
It goes without saying that both configurations violate the
pre-requisites of Theorem~\ref{th:rotation-quadrilateral}.

The following lemma states that the neighbour construction acts
projectively on the set of lines obtained by revolving $M_0$ about
$N_0$. It allows a simplifying assumption during the later
calculations.

\begin{lemma}
  \label{lem:projectivity}
  Consider three lines $N_0$, $N_1$ and $M_0$ such that $M_0$ and
  $N_0$ are skew while $N_0$ and $N_1$ are intersecting. Denote the
  set of lines obtained by rotating $M_0$ about $N_0$ by
  $\lset{M}_0$. Then the neighbour construction with respect to $N_0$
  and $N_1$ (equipped with some orientation) induces a projective
  mapping between $\lset{M}_0$ and its image $\lset{M}_1$ under the
  neighbour construction.  (Figure~\ref{fig:projectivity}).
\end{lemma}

\begin{figure}
  \centering
  \includegraphics{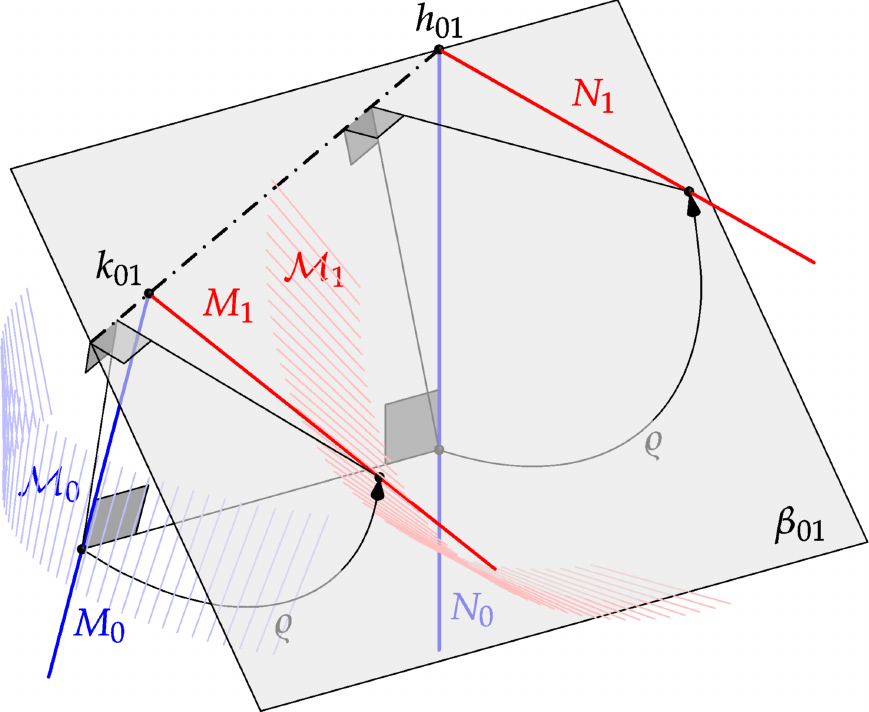}
  \caption{Notation in Lemma~\ref{lem:projectivity}}
  \label{fig:projectivity}
\end{figure}

\begin{proof}
  Our computation is based on a Cartesian coordinate frame $\{o; x, y,
  z\}$ which is projectively extended such that vanishing of the first
  coordinate characterizes ideal points. We choose $N_0$ as the
  $z$-axis and $M_0$ such that the shortest distance between $N_0$ and
  $M_0$ is attained at the pedal points $[1,0,0,0] \in N_0$ and
  $[1,1,0,0] \in M_0$. The ideal point on $M_0$ than equals
  $[0,0,1,v]$ with some $v \in \RSet$. Because of the rotational
  symmetry of the set $\lset{M}_0$, we may assign the coordinates $[s,
  0, t, u]$ to the bisector plane $\beta_{01}$ of $N_0$ and $N_1$.
  Identifying geometric entities with their homogeneous coordinate
  vectors, we can write
  \begin{equation}
    \label{eq:21}
    \begin{gathered}
      N_0 = [0, 0, 1, 0, 0, 0],\
      N_1 = [0, 2tu, u^2-t^2,2st,0,0],\\
      h_{01} = N_0 \cap \beta_{01} = N_0 \cap N_1 = [u, 0, 0, -s],\\
      M_0(\lambda) = [-2\lambda, \lambda^2-1, u(\lambda^2+1),
                      2u\lambda, u(1-\lambda^2), 1+\lambda^2].
    \end{gathered}
   \end{equation}
   As expected, $M_0(\lambda)$ is a quadratic parametrization in
   Plücker coordinates of the line-set $\lset{M}_0$. The theorem's
   statement amounts to saying that $M_1(\lambda)$ is also
   quadratic. The intersection point of $M_0$ with $\beta_{01}$ equals
   \begin{equation}
     \label{eq:22}
     k_{01}(\lambda) =
     \begin{bmatrix}
       uv-t + (t+uv)\lambda^2 \\
       t-uv + 2s\lambda + (t+uv)\lambda^2 \\
       s + 2uv\lambda - s\lambda^2 \\
       -sv-2tv\lambda-sv\lambda^2
     \end{bmatrix}.
   \end{equation}
   Now we compute $M_1(\lambda)$ by reflecting $M_0(\lambda)$, at
   first in the plane $\beta_{01}$ and then in $\gamma_{01} = N_1 \vee
   k_{01}(\lambda)$. We refrain from giving the details of the
   calculation; the relevant formulas are
   \eqref{eq:17}--\eqref{eq:20}. The outcome is a polynomial vector
   function $M_1(\lambda)$ which a priory is of degree 20 in
   $\lambda$. However, all of its entries have the common factor
   \begin{equation}
     \label{eq:23}
     (t^2+u^2)^9(1+\lambda^2)^8G(\lambda)
   \end{equation}
   with the quadratic polynomial
   \begin{equation}
     \label{eq:24}
     G(\lambda) = s^2+t^2-2tuv+u^2v^2+ 4st\lambda + (s^2+t^2+2tuv+u^2v^2)\lambda^2.
   \end{equation}
   This common factor \eqref{eq:23} can be split off so that only a
   quadratic component remains.
\end{proof}

\begin{remark}
  The vanishing of the polynomial \eqref{eq:23} has a geometric
  meaning.
  \begin{itemize}
  \item If $t^2+u^2$ vanishes, the plane $\beta_{01}$ is at infinity
    and the reflection in $\beta_{01}$ becomes undefined.
  \item If $1+\lambda^2$ vanishes, the direction of the revolute axis
    $\beta_{01} \cap \gamma_{01}$ is isotropic and the rotation
    becomes singular. The same is true for the zeros of $G(\lambda)$.
  \end{itemize}
  All of these instances are irrelevant in our setting.
\end{remark}

Lemma~\ref{lem:projectivity} implies that the map
\begin{equation}
  \label{eq:25}
  \eta\colon \lset{M}_0 \to \lset{M}_0,
  \quad
  M_0 \mapsto M^\star_0,
\end{equation}
as a composition of neighbour maps, is projective as well. The
line-set $\lset{M}_0$ is of course the set obtained by subjecting
$M_0$ to all rotations about $N_0$. We conclude that $\eta$
generically has two fixed lines\,---\,those lines of $\lset{M}_0$ that
intersect the $z$-axis. We are interested in configurations where a
non-trivial fixed line $M_0 \in \lset{M}_0$ exists. In this case
$\eta$ is the identity on $\lset{M}$. This can also be phrased as
follows: \emph{If $(q_0,\vec{m}_0)$ is such that the neighbour
  construction is free of contradictions, the same is true for every
  contact element that is obtained from $(q_0,\vec{m}_0)$ by a
  rotation about $N_0$.} This observation gives us one extra degree of
freedom for a simplifying assumption in the

\begin{proof}[Proof of Theorem~\ref{th:baecklund-construction}]
  The basic idea is to show that in the pencil of lines $M_0$ incident
  with $q_0$ and perpendicular to $p_0 \vee q_0$ precisely two lines
  result in a closed chain of neighbouring contact elements obtained
  by successive application of the neighbour construction. These
  solutions necessarily satisfy the condition of the theorem.

  We begin by assigning homogeneous coordinates to the four points
  $p_0$, $p_1$, $p_2$, $p_3$. Without loss of generality we may set
  \begin{equation}
    \label{eq:26}
    p_i = [1+t_i^2, 1-t_i^2, 2t_i, 0],
    \quad
    i = 0,\ldots,3.
  \end{equation}
  The values $t_i$ are in $\RSet \cup \{\infty\}$ and $t_i = \infty$
  corresponds to the point $[1,-1,0,0]$. The following calculations do
  not take into account this exceptional value but could easily be
  adapted to handle this situation. Moreover, we make the admissible
  simplification $t_0 = 0$.

  The bisector planes of the points $p_i$ and $p_j$ are
  \begin{equation}
    \label{eq:27}
    \beta_{ij} = [0,S_{ij},T_{ij}, 0],
    \quad
    i,j = 0,\ldots,3;\ i \neq j
  \end{equation}
  where $S_{ij} = t_i+t_j$ and $T_{ij} = t_it_j-1$. The reflection in
  $\beta_{ij}$ is described by the matrix
  \begin{equation}
    \label{eq:28}
    [B_{ij}]
    \begin{bmatrix}
      S_{ij}^2+T_{ij}^2 & 0                   & 0                 & 0 \\
      0                 & T_{ij}^2 - S_{ij}^2 & -2S_{ij}T_{ij}    & 0 \\
      0                 & -2S_{ij}T_{ij}      & S_{ij}^2-T_{ij}^2 & 0 \\
      0 & 0 & 0 & S_{ij}^2 + T_{ij}^2
    \end{bmatrix}.
  \end{equation}
  Starting with the normal vector $\vec{n}_0 = [0, u, v, 1]$, we
  compute
  \begin{equation}
    \label{eq:29}
    \vec{n}_1 = [B_{01}] \cdot \vec{n}_0,
    \quad
    \vec{n}_3 = [B_{30}] \cdot \vec{n}_0,
    \quad
    \vec{n}_2 = [B_{21}] \cdot \vec{n}_1.
  \end{equation}

  The first point $q_0$ of the Bäcklund transform should be such that
  the connecting line $p_0 \vee q_0$ is perpendicular to
  $\vec{n}_0$. Due to Lemma~\ref{lem:projectivity}, the possible
  choices for $\vec{m}_0$ can be parametrized as
  \begin{equation}
    \label{eq:30}
    \vec{m}_0 = [0, \lambda u, \lambda v, ev+\lambda].
  \end{equation}
  Note that the invalid choice $\vec{m}_0 \parallel \vec{n}_0$
  corresponds to $\lambda \to \infty$ and will automatically drop out
  during our calculations. The second invalid choice is obtained for
  $\lambda = 0$ in which case the line $M_0$ intersects the circle
  axis.

  The point $q_0$ is then
  \begin{equation}
    \label{eq:31}
    q_0 = [1+t_0^2, 1-t_0^2+ev, 2t_0-eu, 0]
  \end{equation}
  where $e$ is a parameter that determines the distance between $p_0$
  and $q_0$.
  
  Using the additional admissible simplification $t_0 = 0$, the
  intersection point $k_{0j} = M_0 \cap \beta_{01}$ becomes
  \begin{equation}
    \label{eq:32}
    k_{0j} =
    \begin{bmatrix}
      (ut_1-v)\lambda \\
      -(v+e(u^2+v^2))\lambda \\
      -(v+e(u^2+v^2))t_1\lambda \\
      -(t_1+t_1e(u+v))(ev+\lambda)
    \end{bmatrix}, \quad j \in \{1,3\}.
  \end{equation}

  By means of \eqref{eq:19} and \eqref{eq:20}, we compute the
  rotations $R_{01}$ and $R_{03}$ to obtain the homogeneous
  coordinates of the points $q_1$, $q_3$ and the vectors (ideal
  points) $\vec{m}_1$, $\vec{m}_3$. Common factors of the shape
  \begin{equation}
    \label{eq:33}
    e^2(u^2+v^2)(1+t_j^2)^4,\quad j \in \{1, 3\},
  \end{equation}
  can be split off. The resulting expressions are just a little too
  long to be displayed here.

  In the same way, we compute the intersection points $k_{12}$ and
  $k_{23}$, where we immediately split off the common factor
  \begin{equation}
    \label{eq:34}
    (1+t_j^2)^2((u^2+v^2+1)(ut_j-v)^2\lambda^2 + v^2(eu + t_j(ev+1))^2),
    \quad
    j \in \{1, 3\},
  \end{equation}
  the rotation $R_{12}$, and the vector $\vec{m}_2$. This latter
  vector can be divided by
  $e^2(u^2+v^2)(1+t_1^2)^6(1+t_2^2)^4Q(\lambda)$
  where
  \begin{multline}
    \label{eq:35}
    Q(\lambda) = (u^2+v^2+1)^2(t_1u-v)^4\lambda^4 +\\
    2v^2(eu+t_1(ev+1))^2(u^2+v^2+1)(t_1u-v)^2\lambda^2 + v^4(eu+t_1(ev+1))^4.
  \end{multline}

  The closure condition of the neighbor construction is now the linear
  dependence of $\vec{m}_2$ and the ideal point (direction vector)
  $[0, \vec{k}_2]$ of $k_{12} \vee k_{23}$. It turns out that
  $\vec{k}_2$ can be divided by
  \begin{equation}
    \label{eq:36}
    (t_1-t_3)(v+e(u^2+v^2))\lambda.
  \end{equation}
  We compute the cross-product $\vec{m}_2 \times \vec{k}_2$ and
  extract the greatest common divisor
  \begin{equation}
    \label{eq:37}
    2evt_2 G_1(\lambda) G_2(\lambda)
  \end{equation}
  of its entries. Here, $G_1(\lambda)$ and $G_2(\lambda)$ are
  quadratic polynomials in $\lambda$.

  One of them, say $G_1$, equals $\lambda^{-1}$ times the homogenizing
  coordinate of $k_{12}$. It is a spurious solution since it vanishes
  in configurations where $\vec{k}_2$ and $\vec{m}_2$ are parallel but
  $M_2$ is different from $k_{12} \vee k_{23}$. The second factor
  equals
  \begin{multline}
    \label{eq:38}
    G_2 = v^2(T_3(e^2u^2+1-e^2v^2)-2T_2e^2uv) +
    + 2ev(T_3(u^2-v^2)-2T_2uv)\lambda\\
    + (T_3(u^2-v^2) - 2uvT_2)(v^2+u^2+1)\lambda^2
  \end{multline}
  where
  \begin{equation}
    \label{eq:39}
    T_2 = t_1t_2-t_1t_3+t_2t_3+1
    \quad\text{and}\quad
    T_3 = t_1t_2t_3+t_1-t_2+t_3.
  \end{equation}
  Its vanishing characterizes the positions of $M_0$ that lead to
  valid solutions. Since \eqref{eq:38} is quadratic in $\lambda$,
  there are precisely two solutions. This finishes the proof of
  Theorem~\ref{th:baecklund-construction}.
\end{proof}

\begin{example}
  \label{ex:discrete-tractrix}
  The maybe simplest example of a pair of Bäcklund mates in the
  continuous settings consists of a straight line $Z$, viewed as a
  degenerate pseudosphere, and the surface of revolution $\Psi$
  obtained by rotating a tractrix with asymptote $Z$ about $Z$. We
  apply our construction to a discrete version of this configuration
  which is actually too irregular to be covered by our
  theory. Nonetheless, we are able to illustrate our basic ideas and
  to recover the geometric essence of the continuous case.

  We define the first principal contact element net $(p,\vec{n})$ as
  follows:
  \begin{equation}
    \label{eq:40}
    \begin{aligned}
      p\colon \ZSet^2 \to \RSet^3,\quad &(i,j) \mapsto (0, 0, j)\\
      \vec{n}\colon \ZSet^2 \to \RSet^3,\quad & \textstyle (i,j)
      \mapsto (0, \cos\bigl(\frac{2i\pi}{k}\bigr),
      \sin\bigl(\frac{2i\pi}{k}\bigr), 0)
    \end{aligned}
  \end{equation}
  The integer $k \in \NSet$ is a shape parameter that affects the
  discretization of the revolute surface.

  The principal contact element net $(p,\vec{n})$ is irregular in the
  sense that the Gaussian curvature of the elementary quadrilaterals
  is undefined. Nonetheless, we can construct the Bäcklund transform
  $(q,\vec{m})$ to $(p,\vec{n})$, defined by the contact element
  $(q_{0,0},\vec{m}_{0,0})$ with $q_{0,0} = (0, d, 0)$, $\vec{m}_{0,0}
  = (0, \cos\alpha, \sin\alpha)$ and $d \in \RSet$. It is depicted in
  Figure~\ref{fig:discrete-tractrix} where we use $\alpha =
  \frac{\pi}{2}$.

  According to our theory, the contact element
  $(q_{0,j+1},\vec{m}_{0,j+1})$ is obtained from
  $(p_{0,j},\vec{n}_{0,j})$, $(p_{0,j+1}, \vec{n}_{0,j+1})$ and
  $(q_{0,j},\vec{m}_{0,j})$, see Figure~\ref{fig:discrete-tractrix},
  left (where we write $p_j$, $q_j$ etc. instead of $p_{0,j}$,
  $q_{0,j}$). It never leaves the plane $x = 0$ and produces a
  discrete curve with vertices $q_{0,j}$ and normal vectors
  $\vec{m}_{0,j}$. The choice $\alpha = \frac{\pi}{2}$ ensures that
  the length of the tangent segment between $q_{0,j}$ and the $z$-axis
  is constant. Hence, we may address the curve as \emph{discrete
    tractrix}.

  The contact element $(q_{i+1,j},\vec{m}_{i+1,j})$ is constructed
  from the contact elements $(p_{i,j},\vec{n}_{i,j})$,
  $(p_{i+1,j},\vec{n}_{i+1,j})$, and $(q_{i,j},\vec{m}_{i,j})$. The
  intersection point of $M_{i,j}$ with the bisector plane of $p_{i,j}$
  and $p_{i+1,j}$ is the ideal point of the $z$-axis so that
  $(q_{i,j},\vec{m}_{i,j})$ and $(q_{i+1,j},\vec{m}_{i+1,j})$
  correspond in a rotation about the $z$-axis through the angle
  $\frac{2\pi}{k}$. The resulting discrete revolute surface can be
  seen in Figure~\ref{fig:discrete-tractrix}, right. Its Gaussian
  curvature is indeed constant and negative. We refrain from deriving
  analytic expressions for describing this discrete surface. These
  have already been established in \cite{bobenko10:_curvature_theory}.
\end{example}

\begin{figure}
  \centering
  \hspace*{0.2\linewidth}
  \includegraphics{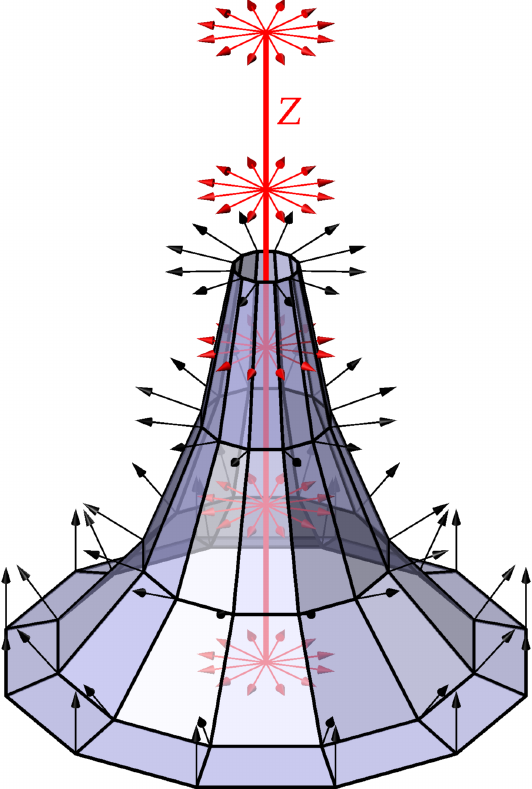}
  \hspace*{-1.0\linewidth}
  \includegraphics{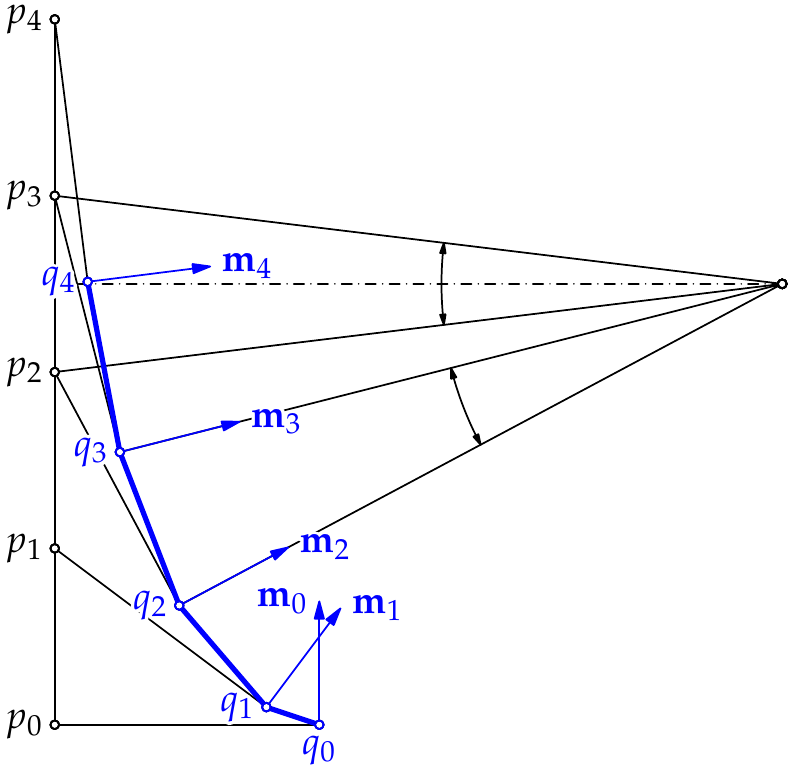}
  \caption{Discrete tractrix and pseudospherical principal contact
    element net of revolution}
  \label{fig:discrete-tractrix}
\end{figure}

\section{Bianchi's Permutation Theorem}
\label{sec:permutation}

Now we prove Bianchi's Permutation
Theorem~\ref{th:bianchi-permutation}, our final major result. We split
our proof into a series of intermediate steps.

\begin{lemma}
  \label{lem:halfturn}
  Assume that $(b,\vec{l})$, and $(c,\vec{m})$ are both Bäcklund
  transforms of the same twist of $(a,\vec{k})$. If the contact
  elements $(b_0,\vec{l}_0)$ and $(c_0,\vec{m}_0)$ correspond to
  $(a_0,\vec{l}_0)$ there exists a half-turn that interchanges
  $(b_0,\vec{l}_0)$ and~$(c_0,\vec{m}_0)$
  (Figure~\ref{fig:pce-completion}).
\end{lemma}

\begin{figure}
  \centering
  \includegraphics{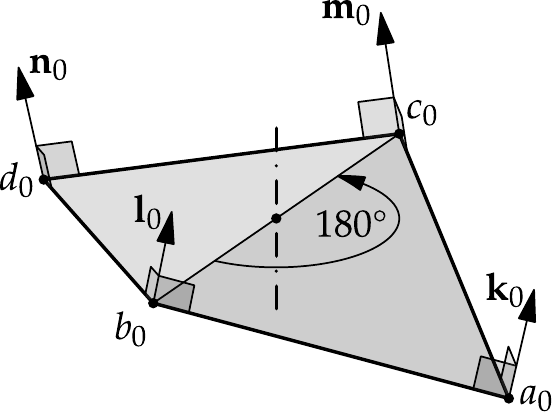}
  \caption{Opposite normals correspond in a half-turn}
  \label{fig:pce-completion}
\end{figure}

\begin{proof}
  In a suitable Euclidean coordinate frame, we have
  \begin{equation}
    \label{eq:41}
    \begin{gathered}
      \vec{l}_0 = (\cos\varphi, \sin\varphi, 0),\quad
      \vec{m}_0 = (\cos\varphi,-\sin\varphi, 0),\\
      b_0 = (0,0, z) + \beta \vec{l}_0,\quad
      c_0 = (0,0,-z) + \gamma \vec{m}_0
    \end{gathered}
  \end{equation}
  with $\varphi \in (0, \frac{\pi}{2})$ and $\beta$, $\gamma \ge
  0$. We try to recover the position of $(a_0,\vec{k}_0)$. It will
  turn out that this is only possible, if $\beta = \gamma$, that is,
  $(b_0,\vec{l}_0)$ and $(c_0,\vec{m}_0)$ correspond in a half-turn.

  Clearly, the possible location of $a_0$ is the intersection line of
  the two tangent planes to $(b_0,\vec{l}_0)$ and
  $(c_0,\vec{m}_0)$. It can be parametrized as
  \begin{equation}
    \label{eq:42}
    a_0(t) = \frac{1}{2\cos\varphi\sin\varphi} (
    (\beta + \gamma)\sin\varphi,
    (\beta - \gamma)\cos\varphi,
    t).
  \end{equation}
  The squared distances form $a_0(t)$ to $b_0$ and $c_0$ equal
  \begin{equation}
    \label{eq:43}
    d_b^2 = (a_0-b_0) \cdot (a_0-b_0),
    \quad
    d_c^2 = (a_0-c_0) \cdot (a_0-c_0),
  \end{equation}
  respectively. (We omit the argument $t$ for sake of readability.)
  Given $a_0$, the unit normal vector $\vec{k}_0$ is obtained as
  $\vec{k}_0 = \vec{k}^\star_0/\Norm{\vec{k}^\star_0}$ where
  \begin{equation}
    \label{eq:44}
    \vec{k}^\star_0 = (a_0 - b_0) \times (a_0 - c_0).
  \end{equation}
  The squared sines of the angles between the contact element normals
  are
  \begin{equation}
    \label{eq:45}
    \sin^2\psi_b = \Norm{\vec{l}_0 \times \vec{k}_0}^2,
    \quad
    \sin^2\psi_c = \Norm{\vec{m}_0 \times \vec{k}_0}^2,
  \end{equation}
  A necessary condition for $a_0(t)$ is now the equality of the ratios
  \begin{equation}
    \label{eq:46}
    \frac{\sin^2\psi_1}{d^2_1} = \frac{\sin^2\psi_3}{d^2_3}
  \end{equation}
  or vanishing of the numerator $P$ of
  \begin{equation}
    \label{eq:47}
    d^2_3\sin^2\psi_1 - d^2_1\sin^2\psi_3.
  \end{equation}
  Generically, $P$ is a polynomial of degree four in $t$. It can be
  factored as
  \begin{equation}
    \label{eq:48}
    P = 4(\beta^2 - \gamma^2)\tan^2\frac{\varphi}{2} P_1(t) P_2(t)
  \end{equation}
  with quadratic polynomials $P_1(t)$ and $P_2(t)$ whose explicit form
  can be readily computed using only rational arithmetic. We are
  rather interested in their discriminants. Using the substitution
  $\varphi = 2\arctan u$ we obtain
  \begin{equation}
    \label{eq:49}
    \begin{gathered}
      \discrim(P_1,t) = -64(1-u^2)^2u^2
      \bigl((\beta - \gamma)(u^4 + 1) + 2(3\gamma+\beta)u^2\bigr)^2,\\
      \discrim(P_3,t) = -64(1-u^2)^2u^2
      \bigl((\beta - \gamma)(u^4 + 1) - 2(3\beta+\gamma)u^2\bigr)^2.
    \end{gathered}
  \end{equation}
  Thus, no real solutions exist for $a_0(t)$ unless $\beta^2 =
  \gamma^2$. This must indeed be the case, since $(a,\vec{k})$ is a
  real solution. Moreover, since $\beta$ and $\gamma$ are not
  negative, we have $\beta = \gamma$ and the half-turn about the first
  coordinate axes indeed interchanges $(b,\vec{m})$ and~$(c,\vec{m})$.
\end{proof}

\begin{lemma}
  \label{lem:completion-uniqueness}
  Consider three contact elements $(a_0,\vec{k}_0)$,
  $(b_0,\vec{l}_0)$, and $(c_0,\vec{m}_0)$ such that $\twist(A_0,B_0)
  = \twist(A_0,C_0)$. Then there exists a unique contact element
  $(d_0,\vec{n}_0)$ such that
  \begin{itemize}
  \item $\dist(b_0,d_0)=\dist(c_0,d_0)=\dist(a_0,b_0)=\dist(a_0,c_0)$,
  \item $\vec{n}_0$ is perpendicular to the vectors that connect $d_0$
    to both, $b_0$ and $c_0$, and
  \item $\twist(B_0,D_0) = \twist(C_0,D_0) = \twist(A_0,B_0) = \twist(A_0,C_0)$.
  \end{itemize}
\end{lemma}

\begin{proof}
  For the time being, we ignore the twist conditions. By the same
  calculation as in the proof of Lemma~\ref{lem:halfturn} (but with
  $\beta = \gamma$) we get distance and angle conditions that
  determine the positions of the contact element $(d_0,\vec{n}_0)$. In
  fact, the possible locus of points $d_0$ is a straight line,
  parametrized by a vector function $d_0(t)$ as in \eqref{eq:42}. The
  unit normal vector $\vec{n}_0(t)$ can be found as in
  \eqref{eq:44}. For all contact elements $(d_0(t), \vec{n}_0(t))$ the
  twist with respect to $(b_0,\vec{l}_0)$ and $(c_0,\vec{m}_0)$ is the
  same.

  Thus, only one twist condition, say $\twist(B_0,D_0) =
  \twist(A_0,B_0)$, is relevant. Its square turns out to be quadratic
  in $t$. Thus, there is exactly one solution to $\twist(B_0,D_0) =
  \pm\twist(A_0,B_0)$ apart from $(a_0,\vec{k}_0)$. It is indeed a
  solution with equal sign, since it can be obtained by rotating
  $(a_0,\vec{k}_0)$ through {180\textdegree} about the half-turn axis
  that interchanges $(b_0,\vec{l}_0)$ and $(c_0,\vec{m}_0)$.
\end{proof}

Lemma~\ref{lem:completion-uniqueness} allows an unambiguous contact
element wise construction of the contact element net $(d,\vec{n})$. It
remains to be shown that $(d,\vec{n})$ is indeed a principal contact
element net. Our proof makes use of a remarkable relation between the
Bäcklund transform and the Bennett linkage that was already noticed by
Wunderlich in \cite{wunderlich51:_differenzengeometrie}. The original
references to the Bennett linkage are
\cite{bennett03:_a_new_mechanism,%
  bennett14:_skew_isogramm_mechanism}. A more accessible description
of its geometry is
\cite[Section~10.5]{hunt78:_kinematic_geometry_mechanisms}.


The Bennett linkage is a spatial four-bar mechanism with a
one-parametric mobility. It consists of four skew revolute axes $K_0$,
$L_0$, $N_0$, $M_0$ such that the normal feet on one axes to the two
neighbouring axes coincide. Denote these points by $a_0$, $b_0$,
$d_0$, and $c_0$, respectively. In order to be mobile, the mechanism
has to fulfill the constraints
\begin{itemize}
\item $\dist(a_0,b_0) = \dist(c_0,d_0)$, $\dist(a_0,c_0) =
  \dist(b_0,d_0)$,
\item $\twist(A_0,B_0) = \twist(C_0,D_0)$, $\twist(A_0,C_0) = \twist(B_0,D_0)$
\end{itemize}
(see for example
\cite[Section~10.5]{hunt78:_kinematic_geometry_mechanisms}). As
suggested by our notation, corresponding normal lines in
Theorem~\ref{th:bianchi-permutation} form the axes of a special
Bennett linkage where all four twists are not only equal in pairs but
equal as a whole. Thus, we may refer the reader to
Figure~\ref{fig:pce-completion} for an illustration of a Bennett
linkage.

A Bennett linkage allows infinitely many incongruent realizations in
space that exhibit the same relative positions, characterized by
normal distance and angle, between any two adjacent axes. If one link
is kept fix, the configuration space of the opposite link has the
topology the projective line (it can be described as a conic on the
Study quadric, see
\cite{husty09:_algebraic_geometry_kinematics}). Hence, any two
realizations of Bennett's linkage can be continuously transformed into
each other without changing the relative position of neighboring
axes. This observation is a key ingredient in the proof of
Theorem~\ref{th:bianchi-permutation}.

As the final preparatory step for the proof of
Theorem~\ref{th:bianchi-permutation}, we mention an obvious
alternative to the characterization of principal contact element nets
in Definition~\ref{def:pcen}. Observe that two contact elements
$(a_0,\vec{k}_0)$ and $(b_0,\vec{l}_0)$ in general position correspond
in a unique rotation. The rotation axis is found as line of
intersection of two bisector planes of corresponding points (for
example of $a_0$, $b_0$ and $a_0+\vec{k}_0$, $b_0+\vec{l}_0$). Two
rotations exist if and only if the bisector planes of two
corresponding points coincide. In this case the bisector planes of all
corresponding points coincide and $(a_0,\vec{k}_0)$ can be rotated in
infinitely many ways to $(b_0,\vec{l}_0)$. The rotation axes lie in
the common bisector plane of corresponding points and are incident
with $K_0 \cap L_0$. This leads to

\begin{proposition}
  \label{prop:pcen}
  A contact element net $(a,\vec{k})$ is a principal contact element
  net if and only if any two neighbouring contact elements
  $(a_i,\vec{k}_i)$ and $(a_j,\vec{k}_j)$ correspond in two (and hence
  in infinitely many) rotations. The rotation axes are incident with
  the intersection point $K_i \cap K_j$ of the two normals and lie in
  the unique bisector plane of $(a_i,\vec{k}_i)$ and
  $(a_j,\vec{k}_j)$.
\end{proposition}

\begin{figure}
  \centering
  \includegraphics{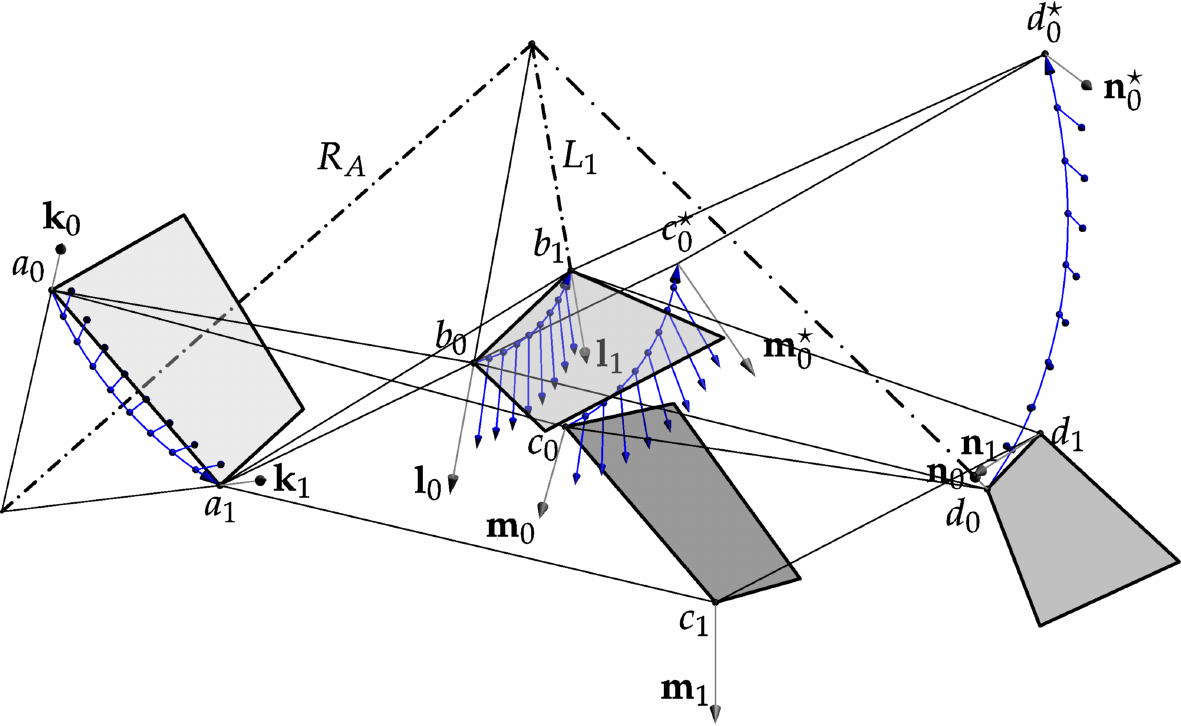}
  \includegraphics{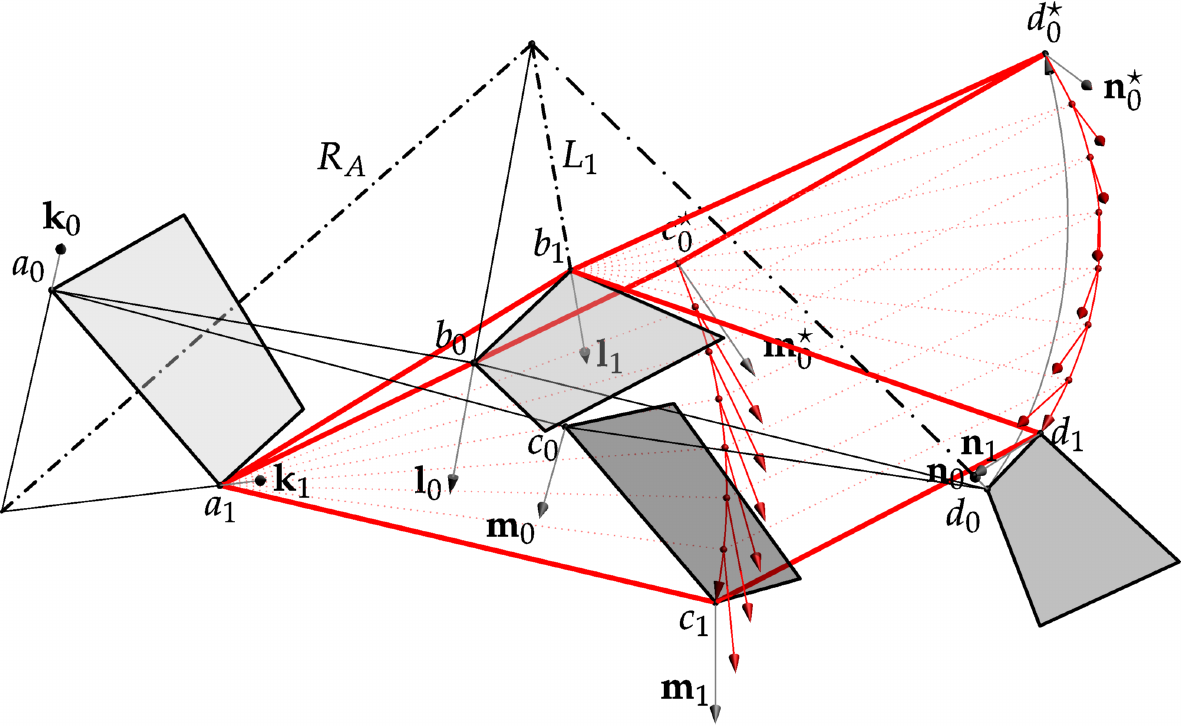}
  \caption{Rotations about $R_A$ (top) and about $L_1$ (bottom) in the
    proof of Theorem~\ref{th:bianchi-permutation}}
  \label{fig:bianchi-rotations}
\end{figure}

\begin{proof}[Proof of Theorem~\ref{th:bianchi-permutation}]
  We consider an elementary quadrilateral $(a_0,\vec{k}_0)$, \ldots,
  $(a_3,\vec{k}_3)$ of the principal contact element net $(a,\vec{k})$
  and the corresponding quadrilaterals $(b_0,\vec{l}_0)$, \ldots,
  $(b_3,\vec{l}_3)$, and $(c_0,\vec{m}_0)$, \ldots, $(c_3,\vec{m}_3)$
  under the two Bäcklund transforms. By
  Lemma~\ref{lem:completion-uniqueness} there exist uniquely
  determined contact elements $(d_0,\vec{n}_0)$, \ldots,
  $(d_3,\vec{n}_3)$ that satisfy all distance and angle constraints
  imposed by the Bäcklund transform. We have to show that these
  contact elements form the elementary quadrilateral of a principal
  contact element net. In view of Proposition~\ref{prop:pcen} it is
  sufficient to show that one neighbouring pair, say $(d_0,\vec{n}_0)$
  and $(d_1,\vec{n}_1)$, corresponds in two different rotations.

  As already argued earlier, there exists a rotation $R$ that maps
  $(a_0,\vec{k}_0)$ to $(a_1,\vec{k}_1)$ and, at the same time,
  $(b_0,\vec{l}_0)$ to $(b_1,\vec{l}_1)$. Its axis $A_R$ is spanned by
  $K_0 \cap K_1$ and $L_0 \cap L_1$
  (Figure~\ref{fig:bianchi-rotations}, top).

  The rotation $R$ transforms $(c_0,\vec{m}_0)$ to a contact element
  $(c^\star_1,\vec{m}^\star_1)$ and $(d_0,\vec{n}_0)$ to a contact
  element $(d^\star_1,\vec{n}^\star_1)$. Since $K_1$, $L_1$,
  $N^\star_0$, $M^\star_0$ and $K_1$, $L_1$, $N_1$, $M_1$ are just two
  realizations of the axes of the same Bennett linkage (with
  respective normal feet $a_1$, $b_1$, $d^\star_0$, $c^\star_0$ and
  $a_1$, $b_1$, $d_1$, $c_1$), we can rotate
  $(d^\star_1,\vec{n}^\star_1)$ about $K_1$ into
  $(d_1,\vec{n}_1)$. (Via Bennett's linkage, this rotation can be
  coupled with a rotation of $c^\star_0$ to $c_1$ about the axis
  $K_1$\,---\,but this is not relevant at this point of our proof.)
  The composition of $R$ with this last rotation
  (Figure~\ref{fig:bianchi-rotations}, bottom), call it $Q$, is again
  a rotation $S$ because the axes $R_A$ and $L_1$ of $R$ and $Q$
  intersect. Thus, we have found one rotation that transforms
  $(d_0,\vec{n}_0)$ to $(d_1,\vec{n}_1)$. If we start with the
  rotation that maps $(a_0,\vec{k}_0)$ to $(a_1,\vec{k}_1)$ and
  $(c_0,\vec{m}_0)$ to $(c_1,\vec{m}_1)$ we obtain a second rotation
  $T$ of $(d_0,\vec{n}_0)$ to $(d_1,\vec{n}_1)$. Generically, the two
  rotations $S$ and $T$ are different, because $L_0$ and $M_0$ do not
  intersect. We conclude that the contact elements $(d_0,\vec{n}_0)$
  and $(d_1,\vec{n}_1)$ satisfy the principal contact element
  criterion of Proposition~\ref{prop:pcen}.
\end{proof}

\section{Conclusion}
\label{sec:conclusion}

In this article we gave a fairly complete geometric treatment of the
Bäcklund transform of principal contact element nets. We proved the
most relevant results (Theorems~\ref{th:gaussian-curvature},
\ref{th:baecklund-construction}, and \ref{th:bianchi-permutation})
which are typical of this curious surface relation. Moreover, we
consider the disclosed relations to discrete kinematics and in
particular to discrete rotating motions to be of interest. In this
context we would emphasize Theorem~\ref{th:rotation-quadrilateral},
which already brings us to open issues. We were unable to provide
certain proofs (those of Theorems \ref{th:rotation-quadrilateral} and
\ref{th:baecklund-construction}) and in a way that make the involved
calculations manually tractable. There is still room left for
simplifications. Moreover, an analytic complement to our geometric
reasoning, maybe in the style of \cite{schief03:_unification}, would
be desirable.



%



\end{document}